\documentclass[12pt,leqno]{article}
\usepackage{graphicx, fullpage}
\usepackage{amsmath,amssymb,amsthm,amscd, bm}
\usepackage{fancyhdr}
\usepackage[mathscr]{eucal}
\usepackage{amsfonts}
\begin{document}
\parskip=6pt

\theoremstyle{plain}

\newtheorem {thm}{Theorem}[section]
\newtheorem {lem}[thm]{Lemma}
\newtheorem {cor}[thm]{Corollary}
\newtheorem {defn}[thm]{Definition}
\newtheorem {prop}[thm]{Proposition}
\numberwithin{equation}{section}

\newcommand{\cE}{{\cal E}}
\newcommand{\cF}{{\cal F}}
\newcommand{\cA}{{\cal A}}
\newcommand{\cC}{{\cal C}}
\newcommand{\cH}{{\cal H}}
\newcommand{\cU}{{\cal U}}
\newcommand{\cK}{{\cal K}}
\newcommand{\cM}{{\cal M}}
\newcommand{\cO}{{\cal O}}
\newcommand{\cN}{{\cal N}}
\newcommand{\cP}{{\cal P}}
\newcommand{\cV}{{\cal V}}
\newcommand{\cW}{{\cal W}}
\newcommand{\bC}{\mathbb C}
\newcommand{\bP}{\mathbb P}
\newcommand{\bN}{\mathbb N}
\newcommand{\bA}{\mathbb A}
\newcommand{\bR}{\mathbb R}
\newcommand{\fg}{\mathfrak g}
\newcommand{\fy}{\mathfrak y}
\newcommand{\fh}{\mathfrak h}
\newcommand{\fV}{\mathfrak V}
\newcommand{\var}{\varepsilon}
\renewcommand\qed{ }
\newcommand{\sgrad}{\text{sgrad} \,}
\newcommand{\sgru}{\text{sgrad}}
\newcommand{\grad}{\text{grad}\,}
\newcommand{\gru}{\text{grad}}
\newcommand{\id}{\text{id}}
\newcommand{\ad}{\text{ad}}
\newcommand{\const}{\text{const}\,}
\newcommand{\Ker}{\text{Ker}\,}
\newcommand{\opartial}{\overline\partial}
\newcommand{\Ree}{\text{Re}\,}
\newcommand{\Imm}{\text{Im}\,}
\newcommand{\an}{\text{an}}
\newcommand{\tr}{\text{tr}}
\newcommand{\wf}{\text{WF}_{\text A}}

\begin{titlepage}
\title{\bf On complex Legendre duality\thanks{This research was done while I enjoyed the hospitality of the  Center for Advanced Study of the Norwegian Academy of Sciences. In addition, it was partially supported by NSF grant DMS--1464150.\newline 2010 Mathematics subject classification 32Q15, 32W20, 53C55}}
\author{L\'aszl\'o Lempert\\ Department of  Mathematics\\
Purdue University\\West Lafayette, IN
47907-2067, USA}
\end{titlepage}
\date{}
\maketitle
\abstract

Complex Legendre duality is a generalization of Legendre transformation from Euclidean spaces to K\"ahler manifolds, that  Berndtsson and collaborators have recently constructed. It is a local isometry of the space of K\"ahler potentials. We show that the fixed point of such a transformation must correspond to a real analytic K\"ahler metric. 

\section{Introduction}

If $X$ is an $n$ dimensional compact complex manifold and $\omega_0$ is a smooth K\"ahler form on it, the space of its (relative) K\"ahler potentials is
$$
\cH=\{u\in C^\infty(X)\colon \omega_u=\omega_0+i\partial\bar\partial u>0\}.
$$
As an open subset of the Fr\'echet space $C^\infty(X)$ of smooth functions $X\to\bR$, it inherits a Fr\'echet manifold structure, and the tangent bundle of this manifold has a canonical trivialization $T\cH\approx \cH\times C^\infty(X)$.
The Mabuchi length $|\xi|_u$ of a tangent vector $\xi\in T_u\cH\approx C^\infty (X)$ is given by
$$
|\xi|_u^2=\int_X \xi^2 \omega_u^n \bigg/\int_X \omega_u^n.
$$
This turns $\cH$ into a smooth Riemannian manifold whose curvature is covariantly constant, see \cite{M}. 
In finite dimensional Riemannian manifolds covariantly constant curvature implies that about every point of the manifold there are local symmetries. 
Such symmetries exist in $\cH$ as well.
Berndtsson, Cordero--Erauskin, Klartag, and Rubinstein in \cite{BCKR} proposed a generalization of the Legendre transformation from Euclidean spaces to K\"ahler manifolds, and this ``complex Legendre duality"  gives rise to local symmetries:

\begin{thm} Suppose $u\in\cH$ and $\omega_u$ is real analytic. Then $u$ has a neighborhood $\cU\subset\cH$ and there is a $C^\infty$ diffeomorphism $F\colon \cU \to\cU$ that is an involution, an isometry of the Mabuchi metric, fixes $u$, and its differential $F_*$  acts on $T_u\cU$ as multiplication by $-1$.
\end{thm}
A variant of this was already known to Semmes, see \cite{Se}. In this paper we show that for such a symmetry $F$ about $u$ to exist it is also necessary that $\omega_u$ be analytic.
If $\cU\subset\cH$ is open, we call a $C^k$ map $F\colon\cU\to\cH$ an isometry if $F(\cU)$ is open, $F\colon\cU\to F(\cU)$ is a diffeomorphism, and
$$
|F_*\xi|_{F(u)}=|\xi|_u\qquad\text{for all}\quad u\in\cU,\quad\xi\in T_u\cH;
$$
we call it a $C^k$ symmetry about $u\in\cU$ if in addition $F(u)=u$ and $F_*\xi=-\xi$ for all $\xi\in T_u\cU$.
\begin{thm} If $F\colon \cU\to \cH$ is a $C^\infty$ symmetry about $u$, then $\omega_u$ is real analytic.
\end{thm}
Let $D\subset\bC$ be smooth Jordan domain, and pull back $\omega_0$ by the projection $D\times X\to X$ to a $(1,1)$ form $\omega$. 
The proof of Theorem 1.2 is obtained through the study of the homogeneous Monge--Amp\`ere equation
\begin{equation}\begin{aligned}
(\omega+i\partial\bar\partial w)^{n+1}&=0\qquad\text{on}\quad \bar D\times X\\
\omega+i\partial\bar\partial w|\{s\}\times X&>0\qquad \text{for all } s\in\bar D,
\end{aligned}\end{equation}
where $\omega$ is assumed analytic. 
The idea is that a symmetry $F$ about $u$ gives rise to a reflection principle.
 Given an analytic arc $I\subset \partial D$, solutions of (1.1) can be continued across $I\times X$ as solutions, provided $w(s,\cdot)=u$ for all $s\in I$. 
In [L3] we imposed  restrictions on the analytic wave front set of solutions of the Monge--Amp\`ere equation in two sided neighborhoods of $I\times X$. 
These restrictions, combined with Donaldson's theorem on solvability of Dirichlet problems for (1.1), imply that in fact the analytic wave front set of $u$  must be empty, and so $u$ must be analytic.

The paper  \cite{L3} studies more general isometries, and most of its results apply to isometries that are merely $C^1$.
 By contrast, we could not prove Theorem 1.2 under the assumption that $F$ is $C^1$, because we do not know whether such symmetries give rise to a reflection principle for the Monge--Amp\`ere equation. 
The gap between $C^1$ and $C^\infty$ is not as wide as it may seem, though: according to \cite[Corollary 5.2]{L2}, $C^2$ isometries are automatically $C^\infty$.

\section{Vector valued Poisson equation}

The link between the symmetry $F$ of Theorem 1.2 and reflection for the Monge--Amp\`ere equation is Donaldson's version of the WZW equation.
 This is a second order partial differential equation for $\cH$ valued functions defined on open subsets of $\bC$. The equation is nonlinear, but its second order part is just the Laplacian.
In this section we quickly verify that basic regularity theory of the Laplacian carries over from scalar valued functions to functions valued in locally convex (topological vector) spaces.

Fix a locally convex space $V$. That is, $V$ is a real vector space on which a family $\cP$ of seminorms induces a topology. 
We assume that the topology is Hausdorff and sequentially complete, like $C^\infty(X)$. 
For basic notions of calculus and geometry in $V$ see e.g. \cite{L2}. Write $V'$ for the space of continuous linear forms $V\to\bR$, and if $l\in V'$ and $p\in\cP$,
$$
p^*(l)=\sup\{l(v)\colon v\in V,\, p(v)\le 1\}\le\infty.
$$

We will work with H\"older classes of  $V$ valued functions. Let $\Omega\subset\bR^m$ be open and $k\in(0,\infty)\setminus\bN$. 
A function $f\colon\Omega\to V$ is $C^k$ if it is bounded, has partials of all orders $\le [k]$ (the integer part of $k$), and any of those partials $g=\partial^\alpha f$  with $|\alpha|\le [k]$ satisfies, with $\{k\}=k-[k]$,
$$
\sup\Big\lbrace\frac{p\big(g(s)-g(t)\big)}{|s-t|^{\{k\}}}\colon s,t\in\Omega,\, s\neq t\Big\rbrace=\tilde p_{\{k\}}(g)<\infty\qquad\text{for all }p\in\cP.
$$
The space of such functions is denoted  $C^k(\Omega, V)$, and the space of functions that are locally such,  $C^k_{\text{loc}}(\Omega, V)$. 
If $0<k<1$, the definition of $C^k$ spaces makes sense for not necesarily open subsets $\Omega\subset\bR^m$. Each seminorm $p\in\cP$ determines a seminorm
$$
p_k(f)=\max\{\sup_\Omega p(f),\tilde p_{\{k\}}(\partial^\alpha f)\colon |\alpha|\le [k]\}
$$ 
on  $C^k(\Omega, V)$. When $V=\bR$ and $p(v)=|v|$, we write, as usual, $|\quad|_k$ for $p_k$.
\begin{lem}A function $f\colon\Omega\to V$ is in  $C^k(\Omega, V)$ if and only if for each $p\in\cP$ there is a $C\in(0,\infty)$ such that for every $l\in V'$
\begin{equation}
l\circ f\in C^k(\Omega,\bR)\quad\text{and}\quad |l\circ f|_k\le Cp^*(l).
\end{equation}
\end{lem}
\begin{proof}We only deal with the `if' implication, the other direction being obvious. Assume (2.1), and first consider $k\in(0,1)$. By the Banach--Hahn theorem, for $p\in\cP$
$$
p\big(f(s)\big)=\sup_l\big|l\big(f(s)\big)\big|\quad\text{ and }\quad p\big(f(s)-f(t)\big)=\sup_l\big| l\big(f(s)\big)-l\big(f(t)\big)\big|,
$$
where $\sup$ is over $l\in V'$ such that $p^*(l)\le 1$. In view of (2.1) these suprema are $\le C$, respectively, $\le C|s-t|^k$, and $f\in  C^k(\Omega, V)$ indeed follows. 

Second consider $k>1$. We claim that (2.1) implies that $f$ has first partials, i.e., $\big(f(s+\lambda t)-f(s)\big)/\lambda$ has a limit as $\lambda\to 0$ for any $t\in\bR^m$ parallel to a coordinate axis. 
Say, $t=(1,0,\ldots,0)$. Given $l\in V'$, with $g=l\circ f$, $s\in\Omega$, and $\lambda,\mu\in \bR$ small we have
\begin{multline*}
\Big|\frac{g(s+\lambda t)-g(s)}\lambda-\frac{g(s+\mu t)-g(s)}\mu\Big|=\\
\Big|\int_0^1\big(\partial_1g(s+\lambda xt)-\partial_1 g(s+\mu xt)\big)\,dx\Big|
\le Cp^*(l)|\lambda-\mu|^{\{k\}}.
\end{multline*}
This implies, again by the Banach--Hahn theorem, that  $\big(f(s+\lambda t)-f(s)\big)/\lambda$ is Cauchy, hence converges, as $\lambda\to 0$. 
To conclude, we note that the partials $\partial_i f$ satisfy $l\circ\partial_i f=\partial_i(l\circ f)\in  C^{k-1}(\Omega, V)$, and the lemma for general $k$ follows by induction on $[k]$.
\end{proof}

With the help of this lemma  it is routine to extend various results for elliptic linear pde's from scalar to vector valued functions. 
When $k=0,1,\ldots,\infty$, we continue to write  $C^k_{\text{loc}}(\Omega, V)$ for the space of functions that have continuous partials up to order $k$.
This is what we will need:
\begin{lem}Let $\Omega\subset \bC$ be open, $f\in C^1_{\text{loc}}(\Omega,V)$ and $g\in  C^k_{\text{loc}}(\Omega, V)$, with nonintegral $k$. 
If $f|(\Omega\setminus\bR)$ is $C_{\text{loc}}^2$ and satisfies $\Delta f=g$ on $\Omega\setminus\bR$, then $f\in  C^{k+2}_{\text{loc}}(\Omega, V)$.
\end{lem}
\begin{proof}It suffices to prove when $\Omega$ is a disc and $f,g$ are $C^1$, respectively $C^k$, on a neighborhood of $\bar \Omega$. Let
$$
\Omega_\pm=\{s\in \Omega\colon\Imm s\gtrless 0\},
$$
either of which may be empty. If $V=\bR$, by Green's identity
$$
\int_{\Omega_+}g(t)\log|s-t|+\int_{\partial \Omega_+}\Big(f(t)\frac{\partial\log|s-t|}{\partial n_t}-\frac{\partial f(t)}{\partial n_t}\log|s-t|\Big)=\begin{cases}2\pi f(s),&\text{if } s\in \Omega_+ \\
0,&\text{if } s\notin \bar \Omega_+ .\end{cases}
$$
The integrals are for the $t$ variable, with respect to area or arc length measure, and $\partial/\partial n_t$ stands for outward normal derivation. 
There is a similar formula for $\Omega_-$. When the two are added, the line integrals over $\bR\cap\partial \Omega_\pm$ cancel, and we obtain
\begin{equation}
2\pi f(s)=\int_\Omega g(t)\log|s-t|+\int_{\partial \Omega}\Big(f(t)\frac{\partial\log|s-t|}{\partial n_t}-\frac{\partial f(t)}{\partial n_t}\log|s-t|\Big)
\end{equation}
when $s\in \Omega\setminus\bR$. By continuity, this holds  on all of $\Omega$. 
On any relatively compact $\Omega'\subset \Omega$ the line integral produces a $C^\infty$ function, all whose H\"older norms are controlled by the $C^1$ norm of $f$. If $0<k<1$, then on $\Omega'$ the $C^{k+2}$ norm of the area integral is also controlled, see \cite[Theorem 4.4]{GT}, and we obtain
\begin{equation}
\big|f|\Omega'\big|_{k+2}\le c_k(|f|_1+|g|_k).
\end{equation}
This holds also for $k>1$. Indeed, if $f$ is known to be $C^\infty$, the estimate (2.3) follows by taking $\partial_i$ of $\Delta f=g$ and arguing by induction; for general $f$ by reducing the estimate to the smooth case via convolution with mollifiers.

This takes care of $V=\bR$. With a general $V$ we have the same estimates for $l\circ f$ in place of $f$, where $l\in V'$. By Lemma 2.1 $f\in   C^{k+2}_{\text{loc}}(D, V)$ follows.
\end{proof}

\section{The WZW equation}

We return to the K\"ahler manifold $(X,\omega_0)$, its space $\cH$ of K\"ahler potentials, and the notation in the Introduction. 
The Levi--Civita connection on $\cH$ allows one covariantly to differentiate vector fields along curves. Suppose $I\subset\bR$ is an interval, $f\in C^1( I,\cH)$, and $\xi$ is a $C^1$ section of $f^*T\cH$, i.e., a $C^1$ map $I\ni t\mapsto \xi(t)\in T_{f(t)}\cH\approx C^\infty(X)$. According to Mabuchi \cite[(2.4.1)]{M}, the covariant derivative  of $\xi$ is
\begin{equation}
\nabla_t\xi(t)=\frac{d\xi(t)}{dt}-\frac12\Big(\grad\frac{du(t)}{dt}\, ,\,\grad\xi(t)\Big)\in C^\infty(X)\approx T_{u(t)}\cH.
\end{equation}
 Here $\gru$, taken on $X$, and the inner product $(\, ,\,)$ are with respect to the K\"ahler metric of $\omega_{f(t)}$.

If $D\subset\bC$ is open and $f\in C^1( D,\cH)$, one can define complex covariant derivatives $\nabla_s,\nabla_{\bar s}$ as well, that act on sections of the complexified bundles $f^*(\bC\otimes T\cH)$.
 Writing $s=\sigma+it$, we have for example $\nabla_s=(1/2)(\nabla_\sigma-i\nabla_t)$. 
For functions with values in $\cH$ we use $\partial_\sigma,\partial_t,\partial_s=(1/2)(\partial_\sigma-i\partial_t)$, etc. to denote partial derivatives with respect to the variables indicated. We will also use subscript notation for partial derivatives.
Finally, if $u\in\cH$,  we let $\{\, ,\,\}_u$ or just $\{\, ,\,\}$ denote the Poisson bracket on $T_u\cH\approx C^1(X)$, or on $\bC\otimes T_u\cH$,  induced by the symplectic form $\omega_u$.

The WZW equation is
\begin{equation}
\Ree\nabla_s\partial_{\bar s}f=\frac i2\{\partial_s f,\partial_{\bar s}f\}_{f(s)}
\end{equation}
for a $C^2$ function $f\colon D\to\cH$. 
This is essentially the same as the equation Donaldson introduced in \cite{D1}, but, due to differing conventions, not quite the same. 
For this reason we quickly rederive its connection with the Monge--Amp\`ere equation on $D\times X$. Any $f\in C^2(D, X)$ defines a $w\in C^2(D\times X)$ by $f(s)=w(s,\cdot)$, and $f$ solves the WZW equation (3.2) if and only if $w$ solves the Monge--Amp\`ere equation
\begin{equation}
(\omega+i\partial\bar\partial w)^{n+1}=0.
\end{equation}
Here, as in the Introduction, $\omega$ is the pullback of $\omega_0$ to $D\times X$.
 To verify the equivalence of (3.2), (3.3) at some $\tilde s\in D$, $\tilde x\in X$, we introduce local coordinates $x_1,\ldots,x_n$ on $X$ centered at $\tilde x$ so that $\omega_0+i\partial\bar\partial f(\tilde s)=i\sum dx_j\wedge d\bar x_j$ at $\tilde x$. 
At $\tilde x$, therefore, the K\"ahler metric is twice the Euclidean metric. 
In some neighborhood of $\tilde x$ we can write $\omega_0=i\partial\bar\partial p$ with a smooth potential $p$, and then (3.3) becomes for $v(s,x)=p(x)+w(s,x)$
\begin{equation}
(\partial\bar\partial v)^{n+1}=0\qquad\text{or} \qquad \det\big(v_{x_j\bar x_k} \big)_{0\le j,k\le n}=0.
\end{equation}
Here the meaning of $x_0$ is $s$. On the one hand, Schur's formula for the determinant of a block matrix, 
$$
\det\begin{pmatrix}A & B \\ C& D\end{pmatrix}=\det(A-BD^{-1}C)\det D,
$$
shows that (3.4) is equivalent, at $(\tilde s,\tilde x)$, to
\begin{equation}
0=v_{s\bar s} -(v_{s\bar x_k})(v_{x_j\bar x_k})^{-1}(v_{x_j\bar s})^T=w_{s\bar s} -\sum_{j=1}^n |w_{ x_j\bar s}|^2.
\end{equation}
Since $2\{f_s,f_{\bar s}\}=i\{f_\sigma,f_\tau\}=\sum_{j=1}^n \big(w_{\bar x_j\sigma} w_{x_j t}-w_{x_j\sigma}w_{\bar x_j t}\big)$ and
$$
|\grad f_\sigma|^2+|\grad f_t|^2=2\sum_{j=1}^n\big(|w_{x_j\sigma}|^2+|w_{ x_j t}|^2\big),
$$
the sum in (3.5) can be written
$$
\frac14\sum_{j=1}^n|w_{x_j\sigma}+iw_{x_j t}|^2=\frac18\big(|\grad f_\sigma|^2+|\grad f_t|^2\big)+\frac i2\{f_s,f_{\bar s}\}.
$$

On the other hand (3.1) gives at $\tilde s$, $\tilde x$
\begin{equation}\begin{aligned}
\Ree\nabla_s\partial_{\bar s}f&=\frac 14(\nabla_\sigma f_\sigma+\nabla_tf_t)=\frac 14(f_{\sigma\sigma}+f_{tt})-\frac 18\big(|\grad f_\sigma|^2+|\grad f_t|^2\big) \\
&=w_{s\bar s}-\sum_{j=1}^n|w_{\bar x_j s}|^2+\frac i2\{f_s,f_{\bar s}\}.
\end{aligned}\end{equation}
Comparing this with (3.5) proves that (3.2) and (3.3) are indeed equivalent.

In \cite{L3} we studied isometries $F\colon\cU\to\cH'$ into the space of K\"ahler potentials on another K\"ahler manifold $(X',\omega_0')$, and proved (assuming $X$ connected) that their differentials either preserve or reverse Poisson brackets: if $u\in\cU$ then
\begin{align*}
&\{F_*\xi,F_*\eta\}_{F(u)}=F_*\{\xi,\eta\}_u\quad\text{ for all }\xi,\eta\in T_u\cU, \quad\text { or} \\
& \{F_*\xi,F_*\eta\}_{F(u)}=-F_*\{\xi,\eta\}_u\quad\text{ for all }\xi,\eta\in T_u\cU,
\end{align*}
see Theorem 1.1 there. If $\cU$ is connected, then of course the same alternative will hold at all $u\in\cU$, and we say that $F$ preserves (reverses) Poisson brackets if the first (second) alternative holds.

\begin{lem}Consider a solution $f$ of the WZW equation (3.2).

(a) If $\phi \colon D'\to D$ is a holomorphic map, $D'\subset\bC$ open, then $f'=f\circ\phi$ also solves the WZW equation.

(b) Let $\cH'$ be the space of K\"ahler potentials of another K\"ahler manifold and $\cU\subset\cH$ a neighborhood of $f(D)$. If $F\in C^2(\cU,\cH')$ is an isometry that preserves Poisson brackets, then $F\circ f$ also solves the WZW equation. If $F$ reverses Poisson brackets, then $s\mapsto F(f(\bar s))$ solves the WZW equation.
\end{lem}
\begin{proof}One way to prove (a) is to note that the  functions $w\in C^2(D\times X)$ and $w'\in C^2(D'\times X)$ associated with $f,f'$ are related by $w'=(\phi\times\id_X)^*w$; hence if 
$(\omega+i\partial\bar\partial w)^{n+1}=0$, then   $(\omega+i\partial\bar\partial w')^{n+1}=0$.
 Part (b) follows because for a $C^1$ vector field $\xi$ along $f$  Levi--Civita covariant differentiation $\nabla'$ on $\cH'$ satisfies
$$
\nabla'_\sigma F_*\xi=F_*\nabla_\sigma\xi,\quad \nabla'_t F_*\xi=F_*\nabla_t\xi,\quad\text{and so } \nabla'_s F_*\xi=F_*\nabla_s\xi.
$$
\end{proof}
Now we come to the reflection principle.
\begin{lem}Suppose there is a $C^\infty$ symmetry about $u\in\cH$.
 Let $D\subset\bC$ be a Jordan domain, $I\subset\partial D$ an open analytic arc, and $f\in C^k_{\text{loc}}(D,\cH)$ a solution of the WZW equation (3.2), $k>2$ not an integer (possibly $k=\infty$). If 
$$
\lim_{s\to s_0} f(s)= u\qquad\text{for all }s_0\in I,
$$
then $f$ continues across $I$ as a $C^k_{\text{loc}}$  solution of the WZW equation.
\end{lem}
\begin{proof}We start by observing that since the symmetry $F$ satisfies  $F_*|T_u\cH=-\id_{T_u\cH}$, it reverses Poisson bracket at $u$, hence in a neighborhood of $u$. 
At the price of shrinking $D$, we can assume that this neighborhood contains $f(D)$.
Because of the Riemann mapping theorem we can also assume that $I\subset\bR$ is a line segment and $D$ is in the upper half plane $\{s\in\bC\colon \Imm s>0\}$.
The lemma for $k=\infty$ will follow once we prove it for $k<\infty$, so we take $k$ finite.
Write $V=C^\infty(X)$.
By H\"older continuity, $f$ and its partials of order $\le 2$ extend continuously to $D\cup I$. Moreover, $f$ thus extended will be twice continuously differentiable on $D\cup I$, and its partials of order $\le 2$ will be in $C^{\{k\}}_{\text{loc}}(D\cup I,V)$. This can be seen, for example, for the $\partial_t$ derivative from
$$
\frac{f(s+i\tau)-f(s)}\tau=\frac 1\tau\int_0^1\frac{df(s+i\lambda\tau)}{d\lambda}d\lambda=\int_0^1\partial_t f(s+i\lambda\tau)\, d\lambda,
$$
where $s\in D\cup I$ and $\tau>0$ is small. Of course, $\partial_\sigma f=0$ along $I$. 

We can further extend $f$ to $D_-=\{s\colon \bar s\in D\}$ by $f(s)=F\big(f(\bar s)\big)$.
This defines $f$ as a $C^{1+\{k\}}_{\text{loc}}$ function on $\Omega=D\cup I\cup D_-$, that is in $C^k_{\text{loc}}(D\cup D_-,V)$, and in view of Lemma 3.1, solves the WZW equation on $D\cup D_-$. In other words, on $D\cup D_-$
\begin{equation}
f_{s\bar s}=\frac18\big(|\grad f_\sigma|^2+|\grad f_t|^2\big)+\frac i2\{f_s,f_{\bar s}\},
\end{equation}
see (3.6).
The right hand side here is in $C_{\text{loc}}^{\{k\}}(\Omega,V)$. 
By Lemma 2.2 therefore $f\in C_{\text{loc}}^{2+\{k\}}(\Omega,V)$, and by continuity $f$ solves the WZW equation on $\Omega$. 
We are done if $2<k<3$. If  $k>3$, we continue: the right hand side of (3.7) is now known to be in $C_{\text{loc}}^{1+\{k\}}(\Omega,V)$, and Lemma 2.2 implies $f\in C_{\text{loc}}^{3+\{k\}}(\Omega,V)$, and so on, until we reach $f\in C_{\text{loc}}^k(\Omega,V)$.
\end{proof}

\section{The initial value problem}

In this section we will study an initial value problem for the WZW equation, in two versions. 
In one, the initial Cauchy data are given on the boundary of a domain $D\subset\bC$, in the other, in the interior. 
Although initial value problems for these equations are ill--posed, we prove that boundary Cauchy data that admit a solution form a large set; and that, loosely speaking, for Cauchy data associated with $u\in\cH$ it makes no difference whether they are imposed in the interior or on the boundary, provided there is a symmetry about $u$.

If $M$ is a finite dimensional smooth manifold and $V$ is a locally convex vector space, there is a locally convex topology on the space $C^\infty(M,V)$ of smooth maps. One takes $k\in\bN$, a compact $K\subset M$ contained in a coordinate neighborhood, and one of the seminorms $p$ defining the  topology of $V$. Using the coordinates around $K$ one then lets
$$
p_{k,K}(f)=\max\{p\big(\partial^\alpha f(x)\big)\colon |\alpha|\le k,\, x\in K\},\qquad f\in C^\infty(M,V);
$$
these seminorms $p_{k,K}$ induce the topology in question. 
When $V=C^\infty(X)$, with  $X$ a compact manifold, the same topology is obtained if we identify $C^\infty(M,V)$ with $C^\infty(M\times X)$ by associating with 
$f\in C^\infty(M,V)$ the function $w\in C^\infty(M\times X)$ given by $w(s,\cdot)=f(s)$, and endow  $C^\infty(M\times X)$ with its usual Fr\'echet space structure.

Given $u\in\cH$ and a smooth open  arc $I\subset \bC$, we consider smoothly bounded Jordan domains $D\subset\bC$ such that $I\subset\bar D$, and $f\in C^\infty(\bar D,\cH)$ that solve the WZW equation
\begin{equation}\begin{aligned}
\Ree\nabla_s\partial_{\bar s} f&=\frac i2\{\partial_s f,\partial_{\bar s}f\}\qquad\text {on } D,\\
 f(s)& = u \phantom{\{\partial_s,\partial_{\bar s}\}}\qquad\text{for all } s\in I.
\end{aligned}\end{equation}
Fix a unit normal vector field $\nu$ along  $I$. Consider furthermore  those smooth maps $\xi\colon I\to T_u\cH\approx C^\infty(X)$ for which there are a $D$ and $f$ as above with $\xi=\partial f/\partial\nu$. 
Denote the set of such $\xi$ by $\cE_u=\cE_u(I)$ and denote by $E_u=E_u(I)\subset\cE_u$ the set of those $\xi$ that can be obtained from a $D$ that itself, and not only its closure, contains $I$.
Lemma 3.2 then gives
\begin{lem}If $I$ is analytic and there is a symmetry $F\colon\cU\to\cH$ about $u$, then $E_u=\cE_u$.
\end{lem}

At the same time $\cE_u$ is large, for any $u$: 
\begin{lem}Let $D$ be a smoothly bounded Jordan domain and $I\subset\partial D$ an open arc whose complement is not just one point. 
For a dense set of $\xi\in C^\infty(I,T_u\cH)$ the problem (4.1) has a solution $f\in C^\infty(\bar D,\cH)$ such that the (outer) normal derivative $\partial f/\partial\nu=\lambda\xi$ on $I$ with some $\lambda\in(0,\infty)$. 
Hence the linear hull of $\cE_u$ is dense in $C^\infty(I,T_u\cH)$.
\end{lem}

The lemma is a consequence of \cite[Theorem 1]{D2}, or rather of its proof. Donaldson works with the Monge--Amp\`ere equation, which is, as we have seen, equivalent to the WZW equation. He considers the Dirichlet problem
\begin{equation}\begin{aligned}
(\omega+i\partial\bar\partial w)^{n+1}&=0\qquad\text{on}\quad \bar D\times X\\
\omega+i\partial\bar\partial w|\{s\}\times X&>0\qquad \text{for all } s\in\bar D \\
w|\partial D\times X &=\varphi
\end{aligned}\end{equation}
and proves that the  boundary conditions $\varphi$ for which (4.2) has a solution $w\in C^\infty(\bar D\times X)$ form an open subset of $C^\infty(\partial D\times X)$.
(He deals with $D$ the unit disc, from which the general case follows by the Riemann mapping theorem.)
\cite{D2} does not give all the details of the proof, but the idea is the following. Fix a point $o\in D$. Any solution $w$ of (4.2) can be reconstructed from a foliation $\cF$ of $\bar D\times X$. 
The leaves are graphs of smooth maps $\rho_x\colon \bar D\to X$, parametrized by $x\in X$, that are holomorphic on $D$ and satisfy $\rho_x(o)=x$ as well as a certain boundary condition determined by $\varphi$.
The map $x\mapsto\rho_x\in C^\infty(\bar D, X)$ is also smooth. 
That these $\rho_x$ can be found is proved by the implicit function theorem, assuming of course that they can be found for some $\varphi^0$ and $\varphi$ is close to $\varphi^0$.
We will choose $\varphi^0(s,\cdot)\equiv u$, when the Monge--Amp\`ere solution is $w^0(s,\cdot)\equiv u$ and $\rho_x^0\equiv x$.
However, before turning to the implicit function theorem, we switch from  the Fr\'echet manifold of smooth maps $\bar D\to X$  to Banach manifolds.
 Donaldson uses manifolds modelled on Sobolev spaces, but  it is equally  convenient to work with   the Banach manifold $A^k$ of H\"older continuous maps $\rho\in C^k(\bar D,X)$ that are holomorphic on $D$. Here $k>0$ is not an integer.
The implicit function theorem provides a neighborhood $\cV_k\subset C^\infty(\partial D\times X)$ of $\varphi^0$ and a $C^\infty$ map
$$
P\colon X\times\cV_k\to A^k
$$
such that the functions $\rho_x=P(x,\varphi)$ satisfy $\rho_x(o)=x$ and the requisite boundary condition determined by $\varphi$.

In fact, $P$ maps into $A^\infty=A^k\cap C^\infty(\bar D, X)$. 
This is so because the boundary condition for $\rho_x$ requires that a certain lift $g_x$ of $\rho_x$ have boundary values on a smooth maximally real submanifold of a complex manifold $W$. 
This implies that $g_x$ and so $\rho_x$ are smooth on $\bar D$, see \cite[Lemme 2, p. 436]{L1}. 
Moreover, $P$ is $C^\infty$ even when viewed as a map with values in the Fr\'echet manifold $A^\infty$. 
For let  $(x^1,\varphi^1)\in X\times\cV_k$. Donaldson shows that $P(x^1,\varphi^1)$ is `superregular'; the implicit function theorem therefore produces for each $l\in(0,\infty)\setminus \bN$ a neighborhood $\cW_l\subset X\times C^\infty(\partial D\times X)$ of $(x^1,\varphi^1)$ and a $C^\infty$ map $Q\in C^\infty(\cW_l, A^l)$
such that the functions $\theta_x=Q(x,\varphi)$ also satisfy $\theta_x(o)=x$ and the boundary condition. 
There is a uniqueness property that implies that $P=Q$ on $\cW_l\cap(X\times\cV_k)$. This shows that $P$, viewed as a map into $A^l$ is smooth in a neighborhood of $(x^1,\varphi^1)$. 
As $(x^1,\varphi^1)$ was arbitrary, $P \colon X\times\cV_k\to A^l$ is smooth, and as $l$ was arbitrary, $P$ is indeed smooth as a map into $A^\infty$.
From $P(\cdot,\varphi)$ one can then reconstruct the solution $w$ of (4.2) simply by extending $\varphi$ harmonically along the graphs of $\rho_x=P(x,\varphi)$, $x\in X$. 
The upshot of all this is the following.

\begin{lem}There are a neighborhood $\cV\subset C^\infty(\partial D\times X)$ of $\varphi^0(s,\cdot)\equiv u$ and a smooth map $\Theta\colon \cV\to  C^\infty(\bar D\times X)$  such that $\Theta(\varphi^0)(s,\cdot)=u$ for all $s\in\bar D$, and $\Theta(\varphi)=w$ for $\varphi\in\cV$ satisfies (4.2).
\end{lem}
\begin{proof}[Proof of Lemma 4.2]We argue in the language of the Monge--Amp\`ere equation.
With a small $\lambda\in\bR$ and a function $\psi\in C^\infty(\partial D\times X)$ that vanishes on $I\times X$  consider the solution $w=w_\lambda$ of (4.2) corresponding to $\varphi(s,x)=u(x)+\lambda\psi(s,x)$.
 For example, $w_0(s,\cdot)=u$ for all $s$. 
By Lemma 4.3 $w_\lambda\in  C^\infty(\bar D\times X)$ depends smoothly on $\lambda$. 
Write $\pi$ for the projection $\bar D\times X\to X$.
Taking $\partial/\partial\lambda$ of (4.2) gives that $\dot w=\partial w/\partial\lambda|_{\lambda=0}\in C^\infty(\partial D\times X)$ solves
\begin{align}
\partial\bar\partial \dot w\wedge(\pi^*\omega_u)^n&=0\qquad\text{on }\bar D\times X \\
\dot w&=\psi\qquad\text{on }\partial D\times X.
\end{align}
The meaning of (4.3) is that $\dot w$ is harmonic along the leaves $D\times \{x\}$, $x\in X$.

Translating back to the WZW language, (4.2) implies that $\xi_\lambda\colon I\to T_u\cH$ given by $\xi_\lambda(s)=(\partial w_\lambda/\partial\nu)(s,\cdot)$ is in $\cE_u$. 
Let $\xi=\lim_{\lambda\to 0}\xi_\lambda/\lambda$, so that 
$\xi(s)=(\partial\dot w/\partial\nu)(s,\cdot)$  for $s\in I$. To finish the proof it will suffice to show that such $\xi$, for all choices of $\psi$, form a dense subset of $C^\infty (I,T_u\cH)$, or equivalently, that solutions $\dot w$ of (4.3), (4.4) produce  $\partial\dot w/\partial\nu$ that are dense in $C^\infty(I\times X)$. 
It will be convenient at this point to normalize $D$ so that $I\subset\bR$. Let 
$$
u_1,\ldots,u_m\in C^\infty(X)\qquad\text{and} \qquad q(s,x)=\sum_1^m u_j(x) s^j.
$$
Clearly, $\dot w=\Imm q$ solves (4.3), (4.4) with $\psi=\Imm q|(\partial D\times X)$. 
For this solution $\partial \dot w/\partial\nu=\partial q/\partial s$ along $I\times X$, and such functions indeed form a dense subset of $C^\infty(I\times X)$.
\end{proof}

\section{The proof of Theorem 1.2}

Analyticity properties of a function are encoded in its analytic wave front. 
Any continuous (or hyper--) function $v$ defined on a finite dimensional real analytic manifold $Y$ has an analytic wave front set
$$
\wf(v)\subset T^*Y\setminus\text{zero section}.
$$
Rather than a precise definition---which can be found in \cite[Chapter VII and IX]{H}, \cite[Chapter 6]{Sj}, or \cite[Definition 2.3]{L3}---we will need three facts about wave front sets.
First, $v$ is analytic if and only if $\wf(v)=\emptyset$, see \cite[Theorem 8.4.5]{H} or \cite[Th\'eor\`eme 6.3]{Sj}.
Second, if $p\colon Y\times\bR\to Y$ is the projection, then $\wf(p^*v)=p^*\wf(v)$. This can be read off from Sj\"ostrand's definition, or from the less general \cite[Definition 2.3]{L3}. 
Last,
\begin{lem}Let $X$ be a complex manifold,  $x\in X$, $s\in\bR$, and $U\subset \bC\times X$ a neighborhood of $(s,x)$. 
Suppose $\cF$ is a smooth foliation of $U$ by Riemann surfaces, and the leaf $L$ through $(s,x)$ is transverse to  $H=\bR\times X$. 
If $v\in C(U\cap H)$ can be extended to a function $\tilde v\in C(U)$ that is harmonic along the leaves of $\cF$, and $b\in\wf(v)\cap T^*_{(s,x)}H$, then $b|T_{(s,x)}(L\cap H)=0$.
\end{lem}
This is a weakened variant of \cite[Lemma 2.5]{L3}. That lemma applies when $X$ is an open subset of $\bC^n$ but, since wave fronts are locally determined, the difference is irrelevant.

We return to our K\"ahler manifold $(X,\omega_0)$. Given a $u\in\cH$, we denote by $\sgru_u$  symplectic gradient on $X$ with respect to the symplectic form $\omega_u$.
That is, $\omega_u$ determines an isomorphism $T^*X\to TX$, and for $\eta\in C^\infty(X)$ the vector field $\sgru_u\eta$ is the image of $d\eta$ under this isomorphism. 
We will denote the pairing between $T^*X$ and $TX$ by $\langle\, ,\,\rangle$.

\begin{lem}Assume $\omega_0$ analytic. Let $I\subset\bR$ be an open interval, $D\subset\bC$ its neighborhood, and $f\in C^\infty(D,\cH)$ a solution of the initial value problem (4.1) for the WZW equation. Let furthermore $\xi=\partial f/\partial\nu|I$. If $x\in X$ and $a\in\wf(u)\cap T_x^*X$, then $\langle a,\sgru_u\xi(s,x)\rangle=0$ for every $s\in I$.
\end{lem}
\begin{proof}The statement is a slight extension of \cite[Theorem 2.4]{L3}. 
As there, we let $w(s,\cdot)=f(s)$ and rewrite the WZW equation as $(\omega+i\partial\bar\partial w)^{n+1}=0$.
 Near $x\in X$ we can find an analytic potential $p$ for $\omega_0$; then $\tilde v(s,\cdot)=p+w(s,\cdot)$ satisfies $(\partial\bar\partial \tilde v)^{n+1}=0$. 
According to Bedford and Kalka \cite{BK} there is a foliation $\cF$ as described in Lemma 5.1. Let $L$ be the leaf through $(s,x)$. 
In the proof of \cite[Theorem 2.4]{L3} we computed that $T_{(s,x)}(L\cap H)$ is spanned by
$$
4\Ree\frac\partial {\partial s}-\sgru_u\xi(s,x)\in T_s\bR\oplus T_xX\approx T_{(s,x)}(\bR\times X).
$$
Let $v=\tilde v|I\times X$, so $v(s,\cdot)=u$ for all $s\in I$. 
The transformation property of wave fronts under a projection implies that the projection $I\times X\to X$ pulls back $a\in\wf(u)$ to a $b\in\wf(v)\cap T^*_{(s,x)}(I\times X)$. 
Therefore the conclusion of Lemma 5.1, $b|T_{(s,x)}(L\cap H)=0$, implies $\langle a,\sgru_u\xi(s,x)\rangle=0$.
\end{proof}
\begin{proof}[Proof of Theorem 1.2]According to \cite[Proposition 2.1]{L3} $\cH$ contains a $v$ for which $\omega_v$ is analytic. 
Therefore we lose no generality by  assuming that $\omega_0$ itself is analytic. In this case  we want to show that if there is a symmetry about $u$, then $u$ is analytic, i.e., $\wf(u)=\emptyset$. 
Take an $x\in X$ and $a\in T_x^*X\setminus\{0\}$, and fix an interval $I\subset\bR$. Since $E_u(I)=\cE_u(I)$, and the linear hull of the latter is dense (see Lemmas 4.1, 4.2), for any $s\in I$ there is a $\xi\in E_u$ such that
$$
\langle a,\sgru_u\xi(s,x)\rangle\neq 0.
$$
Now $\xi\in E_u$ means that on some neighborhood $D\subset\bC$ of $I$  the initial value problem (4.1) for the WZW equation has a solution satisfying $\xi=\partial f/\partial\nu|I$. 
Lemma 5.2 therefore implies $a\notin\wf(u)$. Since $a$ was arbitrary, this means $u$ is analytic, as claimed.
\end{proof}

\end{document}